\documentclass{amsart}
\usepackage{a4wide,tikz,hyperref}
\title{On the ergodic theory of Tanaka--Ito type $\alpha$-continued fractions}

\author{Hitoshi Nakada}
\address{Department of Mathematics, Keio University, Yokohama, Japan}
\email{nakada@math.keio.ac.jp}
\author{Wolfgang Steiner}
\address{Universit\'e de Paris, IRIF, CNRS, F--75013 Paris, France}
\email{steiner@irif.fr}

\thanks{This work was supported by the Agence Nationale de la Recherche, project CODYS (ANR-18-CE40-0007).}

\subjclass[2010]{11K50, 11J70}

\newtheorem{thm}{Theorem}

\newtheorem{lem}{Lemma}
\newtheorem{prop}{Proposition}
\newtheorem{cor}{Corollary}

\theoremstyle{remark}
\newtheorem*{rem}{Remark}

\newcommand{\confrac}[2]{
\frac{\displaystyle{
\strut\hfill{#1}\hfill\;\vrule}}
{\displaystyle{
 \strut\vrule\;\hfill{#2}\hfill}}}

\begin{document}
\begin{abstract}
We show the ergodicity of Tanaka--Ito type $\alpha$-continued fraction maps and construct their natural extensions.
We also discuss the relation between entropy and the size of the natural extension domain.
\end{abstract}
\maketitle
\section{Introduction and main results} 
In 1981, two types of $\alpha$-continued fraction maps were defined by \cite{Na-1, T-I}: 
For $\alpha \in [0,1]$,   
\begin{itemize}
\item
the first author considered in \cite{Na-1} the map
\begin{equation} \label{HN}
T_{\alpha}(x) = \left| \frac{1}{x} \right| -  
\left\lfloor \left| \frac{1}{x} \right| +1 - \alpha 
\right\rfloor,
\end{equation}
\item
S. Tanaka and S. Ito \cite{T-I} studied
\begin{equation} \label{T-I}
T_{\alpha}(x) =  \frac{1}{x}  -  
\left\lfloor \frac{1}{x} +1 - \alpha \right\rfloor,
\end{equation} 
\end{itemize}
where $0 \ne x \in [\alpha -1, \alpha)$ and $T_{\alpha}(0) = 0$.

The main aim of these papers was the derivation of the density functions of the absolutely continuous invariant measure by constructing the natural extension of a 1-dimensional continued fraction map as a planar map. 
For the map \eqref{T-I}, this was successful only for $\frac{1}{2} \le \alpha \le \frac{\sqrt{5}-1}{2}$, though for $\frac{1}{2}\le \alpha \le 1$ in case~\eqref{HN}.
In \cite{K-S-S}, this was extended to all $\alpha \in (0,1]$ in case~\eqref{HN}.
Here, we show that this method also works for $\alpha>\frac{\sqrt{5}-1}{2}$ in case~\eqref{T-I}. 
In the sequel, the map $T_{\alpha}$ denotes the second type in the above, except where specified otherwise.
Then $T_{\alpha}$ is symmetric w.r.t.\ $\frac{1}{2}$.
Therefore, we can assume that $\frac{1}{2} \le \alpha \le 1$, and it is easy to extend our results to $0\le \alpha \le \frac{1}{2}$.  
Since there were no proofs of the existence of the absolutely continuous invariant measure for $\alpha>\frac{\sqrt{5}-1}{2}$ and for the ergodicity w.r.t.\ this measure for $\alpha > \frac{1}{2}$, we give these proofs for all  $\alpha$ in $[\frac{1}{2},1]$.  

In \S2, we give some basic properties of~$T_{\alpha}$, in particular that the set of full cylinders generates the Borel algebra (Proposition~\ref{p:full}).  
In \S3, we show the existence of the absolutely continuous invariant probability measure $\mu_{\alpha}$ for $T_{\alpha}$ by  the classical method based on Proposition~\ref{p:full}. 
Then we show its ergodicity. 

\begin{thm} \label{t:ergodic}
There is an ergodic invariant probability measure $\mu_\alpha$ for the dynamical system $([\alpha-1,\alpha),T_{\alpha})$ which is equivalent to the Lebesgue measure.  
\end{thm}

Recall that an ergodic measure preserving map $\hat{S}$ is the natural extension of an ergodic measure preserving map $S$ if $\hat{S}$ is invertible and any invertible extension of $S$ is an extension of $\hat{S}$.
We give the natural extension of $T_{\alpha}$ as a planar map
\[
\mathcal{T}_\alpha(x, y) = \bigg(\frac{1}{x} - \bigg\lfloor \frac{1}{x} + 1-\alpha\bigg\rfloor, \frac{1}{y+\big\lfloor \frac{1}{x} + 1-\alpha\big\rfloor}\bigg),
\]
with $\mathcal{T}_\alpha(0,y) = (0,0)$, and the natural extension domain
\[
\Omega_\alpha = \bigcup_{n\ge0} \overline{\mathcal{T}_\alpha^n\big([\alpha{-}1, \alpha) \times \{0\}\big)}.
\]
Then $\frac{\mathrm{d}x\,\mathrm{d}y}{(1 + xy)^{2}}$ gives an absolutely continuous invariant measure $\hat{\mu}$ of $(\Omega_{\alpha}, \mathcal{T}_\alpha)$, and we denote by $\hat{\mu}_\alpha$ the corresponding probability measure.
The main problem here is to show that $\Omega_{\alpha}$ has positive Lebesgue measure. 
We show the following theorem, where the density function of $\mu_\alpha$ is given by    
\[
\frac{1}{\hat{\mu}(\Omega_{\alpha})} \int_{y :\, (x, y) \in \Omega_{\alpha}} 
\frac{1}{(1 + xy)^{2}}\, \mathrm{d}y.
\]

\begin{thm} \label{t:natext}
For $\alpha \in (g,1]$, $\Omega_{\alpha}$ has positive Lebesgue measure and thus $\left( \Omega_{\alpha}, \mathcal T_{\alpha}, \hat{\mu}_\alpha\right)$ is a natural  extension of $([\alpha-1,\alpha), T_{\alpha}, \mu_\alpha)$. 
\end{thm}

We note that the existence of $\mu_{\alpha}$ follows directly from the result in \S4 but we need the ergodicity proved in Theorem~\ref{t:ergodic} for the concept of a natural extension.

In \S5, we give a selfcontained proof that Rokhlin's formula 
\[
h(T_{\alpha}) = \int_{[\alpha-1, \alpha)} - 2 \log |x|\,  
\mathrm{d}\mu_\alpha
\] 
holds for~$T_{\alpha}$ (Proposition~\ref{p:rokhlin}); we refer to \cite{Zwei} for the general case of one dimensional maps. 
In this paper, we use Propositions~\ref{p:full} and~\ref{p:bound} with the Shannon--McMillan--Breiman--Chung theorem; see \cite{Br, Ch}. 
Moreover, we show that 
\[
-2 \lim_{n\to\infty} \frac{1}{n} \log |q_{\alpha,n}(x)|  = h(T_\alpha)
\]
for almost all $x \in [\alpha-1,\alpha)$, where $q_{\alpha,n}(x)$ is the denominator of the $n$-th convergent of $x$ given by~$T_\alpha$; note that Tanaka and Ito~\cite{T-I} mentioned this fact for $\alpha = 1/2$.

The behavior of the entropy as a function of $\alpha$ will be discussed in the forthcoming paper~\cite{C-L-S}. 
In the case of $T_\alpha$ defined by~\eqref{HN}, it was shown in \cite[Theorem~2]{K-S-S} that $h(T_\alpha) \hat{\mu}(\Omega_\alpha) = \pi^2/6$ for all $\alpha \in (0,1]$, where $\hat{\mu}$ is the invariant measure of the natural extension given by $\frac{\mathrm{d}x\,\mathrm{d}y}{(1 + xy)^{2}}$ (without normalization).
For $T_\alpha$ defined by~\eqref{T-I}, this does not hold: for $\alpha = 1$, the maps defined by \eqref{HN} and \eqref{T-I} are equal and we have thus $h(T_1) \hat{\mu}(\Omega_1) = \pi^2/6$ in both cases; for $\alpha = 1/2$, the maps defined by \eqref{HN} and \eqref{T-I} produce the same continued fraction expansions and have thus the same entropy, but $\Omega_{1/2}$ for \eqref{T-I} is equal to $\Omega_{1/2} \cup (-\Omega_{1/2})$ for~\eqref{HN}, hence we have $h(T_{1/2}) \hat{\mu}(\Omega_{1/2}) = \pi^2/3$ in case~\eqref{T-I}.
For case~\eqref{T-I}, we have the following.

\begin{thm} \label{t:hmu}
The function
\[
\alpha \mapsto h(T_\alpha)\, \hat{\mu}(\Omega_\alpha)
\]
is a monotonically decreasing function of $\alpha \in [\frac{1}{2},1]$. 
\end{thm}

\section{Some definitions and notation}
We start with basic definitions. 
Since we discuss a fixed $\alpha$, we omit $\alpha$ from the index. 
We~define 
\[
a_{k}(x) = \left\lfloor \frac{1}{T_{\alpha}^{k-1}(x)} + 1 -\alpha \right\rfloor , \quad k \ge 1,
\]
when $T_{\alpha}^{k-1}(x) \ne 0$. 
We put $a_{k}(x) = 0$ if $T_{\alpha}^{k-1}(x) = 0$. 
Then we have 
\[
x = \confrac{1}{a_{1}(x)} + \confrac{1}{a_{2}(x)} + 
    \cdots + \confrac{1}{a_{n}(x)} + \cdots,
\]
and the right hand side terminates at some positive integer $n$ if and only if $x$ is a rational number. 
As usual we put 
\begin{equation} \label{q-n}
\begin{pmatrix}
p_{n-1}(x) & p_{n}(x) \\q_{n-1}(x) & q_{n}(x)
\end{pmatrix}
= 
\begin{pmatrix} 
0 & 1 \\ 1 & a_{1}(x)  
\end{pmatrix} 
\begin{pmatrix} 
0 & 1 \\ 1 & a_{2}(x)  
\end{pmatrix} 
\cdots 
\begin{pmatrix} 
0 & 1 \\ 1 & a_{n}(x)  
\end{pmatrix},
\end{equation}
when $a_{n}(x) \ne 0$. 
It is well-known that 
\[
\frac{p_{n}(x)}{q_{n}(x)} = \confrac{1}{a_{1}(x)} + \confrac{1}{a_{2}(x)} + \cdots + \confrac{1}{a_{n}(x)}
\]
and we call $\frac{p_{n}(x)}{q_{n}(x)}$ the $n$-th convergent of the $\alpha$-continued fraction expansion of $x$. 
It is easy to see that $T_{\alpha}^{n}(x)$ is a linear fractional transformation defined by the inverse of \eqref{q-n}, where \eqref{q-n} is the same matrix for all $x$ in the 
same cylinder set of length~$n$.  
Then we see that
\begin{equation} \label{e:rf}
p_n(x) = a_n(x) p_{n-1}(x) + p_{n-2}(x), \qquad q_n(x) = a_n(x) q_{n-1}(x) + q_{n-2}(x), 
\end{equation}
\begin{equation} \label{l-i}
x = \frac{p_{n-1}(x) T_{\alpha}^{n}(x) + p_{n}(x)}{q_{n-1}(x) T_{\alpha}^{n}(x) + q_{n}(x)},
\end{equation}
and
\[
\left| x - \frac{p_{n}(x)}{q_{n}(x)}\right| = 
\left| \frac{T_{\alpha}^n(x)}{q_{n}(x)\cdot (q_{n-1}(x)T_{\alpha}^{n}(x) + q_{n}(x))} \right|;
\]
here we note that the determinants of all matrices in \eqref{q-n} are $\pm 1$. 

In general we use the notation $\begin{pmatrix}p_{n-1} & p_{n} \\q_{n-1} & q_{n}\end{pmatrix}$ without $x$ when $a_{1},\, \ldots,\, a_{n}$ is given without~$x$. 
For a given sequence of non-zero integers, $a_{1}, \, a_{2}, \, \ldots , \, a_{n}$, we denote by $\langle  a_{1}, \, a_{2}, \, \ldots , \, a_{n} \rangle$ the associated cylinder set, i.e.,
\[
\langle  a_{1}, \, a_{2}, \, \ldots , \, a_{n} \rangle = \{ x \in [\alpha -1, \, \alpha): \, a_{1}(x) = a_{1}, \ldots ,a_{n}(x) = a_{n}\} .  
\]
A~sequence $a_{1}, \, a_{2}, \, \ldots , \,a_{n}$ is said to be admissible if the associated cylinder set has an inner point; here we note that any cylinder set is an interval.  
A~cylinder set is said to be full if 
\[
T_{\alpha}^{n}(\langle a_{1}, a_{2}, \ldots , a_{n} \rangle) = [\alpha-1, \alpha) .  
\]
Because of the definition \eqref{q-n} we see that 
\[
\frac{q_{n-1}(x)}{q_{n}(x)} = \confrac{1}{a_{n}(x)} + \confrac{1}{a_{n-1}(x)} + \cdots + \confrac{1}{a_{1}(x)} . 
\]
We set
\[
g = \frac{\sqrt{5}-1}{2}.
\]

\begin{lem}\label{size-lemma}
For any cylinder set $\langle a_{1}, a_{2}, \ldots , a_{n} \rangle$, we have
\[
\lambda(\langle a_{1}, a_{2}, \ldots , a_{n} \rangle) \le g^{-2(n-1)}/2,
\]
where $\lambda$ denotes the Lebesgue measure. 
\end{lem}

\begin{proof}
For $|x| \le g$, we have $|T_\alpha'(x)| = \frac{1}{x^2} \ge \frac{1}{g^2}$. 
For $x \ge g$, we have $\big|(T_\alpha^{2})'(x) \big| = \frac{1}{(x\, T_\alpha(x))^2} \ge 
\frac{1}{g^4}$.
Since the cylinder of length $0$ has measure $1$ and each cylinder of length $1$ has measure at most~$1/2$, this shows the assertion of this lemma. 
\end{proof}

\begin{prop} \label{p:full}
The set of full cylinders generates the Borel algebra of $[\alpha -1, \alpha)$. 
\end{prop}

\begin{proof}
Fix $n \ge 1$. If 
\begin{equation} \label{non-full} 
T_{\alpha}^{k}(\langle a_{1}, a_{2}, \ldots , a_{k} \rangle ) \ne [\alpha-1, \alpha) \quad 
\mbox{for all} \ 1 \le k \le n , 
\end{equation}
then $(a_{1}, a_{2}, \ldots , a_{n})$ is a concatenation of sequences of the form
$(a_{1}(\alpha), a_{2}(\alpha), \ldots , a_{j}(\alpha))$ or $(a_{1}(\alpha{-}1), a_{2}(\alpha{-}1), \ldots , a_{j}(\alpha{-}1))$, $1 \le j \le n$.  
This implies that the number of admissible sequences satisfying
\eqref{non-full} is at most~$2^{n}$.  
We put 
\[
B_{n} = \bigcup_{(a_{1}, \ldots, a_{n})\, \text{with} \, \eqref{non-full}}
\langle a_{1}, a_{2}, \ldots , a_{n} \rangle
\quad \mbox{and} \quad  B = \bigcap_{n=1}^{\infty} B_{n}.
\]
From Lemma~\ref{size-lemma}, we have 
\begin{equation} \label{non-full-size}
\lambda(B_{n}) \le (2g^2)^{-n+1}/4,
\end{equation}
and then $\lambda(B) = 0$ since $2 g^2 < 1$. 
Then we see that
\[
\lambda\bigg(\bigcup_{n=1}^{\infty} T_{\alpha}^{-n}(B)\bigg) = 0. 
\]
This implies that for a.e.\ $x \in [\alpha-1, \alpha)$ we have $T_{\alpha}^{n}(x) \notin B$ for all $n \ge 1$, hence there exists a sequence $n_{1}< n_{2}< \cdots$ (depending on~$x$) such that $T_{\alpha}^{n_{k}}(\langle a_{1}(x), a_{2}(x), \ldots, 
a_{n_k}(x)\rangle)$ is a full cylinder for any $k \ge 1$.     
This shows the assertion of this proposition.  
\end{proof}

The following lemma is essential in this paper. 
\begin{lem} \label{q-n-2}
Let $(a_{1}, \ldots, a_{n})$ be an admissible sequence.
If $\alpha \in [\frac{1}{2},g]$, then we have $|q_n| > |q_{n-1}|$.
If $\alpha \in (g,1]$, then we have $-\frac{1}{2} < \frac{q_{n-1}}{q_n} < 2$, with $\frac{q_{n-1}}{q_n} \ge 1$ only if $a_n = 1$.
\end{lem} 

\begin{proof} 
We proceed by induction on~$n$. 
Since $q_0=1$, $q_1=a_1$, $|a_n| \ge 2$ when $\alpha \in [\frac{1}{2},g]$, $a_n \ge 1$ or $a_n \le -3$ when $\alpha \in (g,1]$,  the statements hold for $n=1$. 
Suppose that they hold for $n-1$ and recall that $\frac{q_n}{q_{n-1}} = a_n + \frac{q_{n-2}}{q_{n-1}}$ by~\eqref{e:rf}.

If $\alpha \in [\frac{1}{2},g]$, then $|a_n| \ge 2$ gives that $|\frac{q_n}{q_{n-1}}| > 1$; see also \cite[Remark 2.1]{T-I}. 

Let now $\alpha \in (g,1]$.
If $a_n < 0$, then we have $a_n \le -3$ and $a_{n-1} \ne 1$, thus $\frac{q_n}{q_{n-1}} < -2$.
If $a_n > 0$, then we have $\frac{q_n}{q_{n-1}} > \frac{3}{2}$ when $a_n \ge 2$, and $\frac{q_n}{q_{n-1}} > \frac{1}{2}$ when $a_n = 1$.
\end{proof}

We now define the jump transformation of~$T_{\alpha}$, which we will use to show the existence of the absolutely continuous invariant measure. 
From Proposition~\ref{p:full}, for a.e.\ $x \in [\alpha - 1, \alpha)$ there exists $n \ge 1$ such that $T_{\alpha}^{n}\langle a_{1}(x), \ldots , a_{n}(x)\rangle = [\alpha-1, \alpha)$.
We denote the minimum of those~$n$ by~$N(x)$. 
If there is no such~$n$, then we put $N(x) = 0$. 
The jump transformation of $T_{\alpha}$ is
\[
\overset{\circ}{T}_{\alpha}:\, [\alpha -1, \alpha) \to [\alpha -1, \alpha),\quad x \mapsto T_{\alpha}^{N(x)}(x).
\]
Note that $y \in \langle a_{1}(x), \ldots , a_{n}(x)\rangle$ means that $a_{j}(y) = a_{j}(x)$ for all $1 \le j \le n$.  
Hence we see that $N(y) = N(x)$. 
Thus there exists a countable partition $\mathcal{J} = \{J_{k}\,:\, k \ge 1\}$ of $[\alpha-1, \alpha)$ such that each $J_{k}$ is a cylinder set of length $N_{k}$ with
$\overset{\circ}{T}_{\alpha}(x) = T_{\alpha}^{N_{k}}(x)$ for $x \in J_{k}$ and $T_{\alpha}^{j} J_k \ne [\alpha-1, \alpha)$, $1 \le j < N_{k}$, $T_{\alpha}^{N_{k}} J_k =[\alpha-1, \alpha)$.  
Obviously, $\overset{\circ}{T}_{\alpha}$ is a piecewise linear fractional map of the form 
\[
\frac{q_{N_{k}}x - p_{N_{k}}}{-q_{N_{k}-1}x + p_{N_{k}-1}}
\]
for $x \in J_{k}$, and it is bijective from $J_{k}$ to $[\alpha-1, \alpha)$. 

\section{Existence of the absolutely continuous 
invariant measure and ergodicity}
We first prove the following. 

\begin{prop} \label{bdd-dis}
For any admissible sequence $a_{1}, \ldots , a_{n}$, 
\[
\frac{1}{9q_{n}^{2}} < \left| \psi_{a_{1}, \ldots , a_{n}}'(y)\right| < \frac{1}{g^4 q_{n}^{2}} 
\]
holds for all $y \in T_{\alpha}^{n} \langle a_{1}, \ldots , a_{n} \rangle$, where $\psi_{a_{1}, \ldots , a_{n}}$ is the local inverse of $T_{\alpha}^{n}$ restricted to $\langle a_{1}, \ldots , a_{n} \rangle$. 
\end{prop}

\begin{proof}
From \eqref{l-i}, we see that
\[
\psi_{a_{1}, \ldots , a_{n}}(y) = 
\frac{p_{n-1}\, y + p_{n}}{q_{n-1}\, y + q_{n}}
\]
and then 
\[
\left| \psi_{a_{1}, \ldots , a_{n}}'(y) \right| = 
\frac{1}{(q_{n-1}y + q_{n})^{2}} 
\]
for $y \in T_{\alpha}^n\langle a_{1}, \ldots , a_{n}\rangle$. 
If $\alpha \in [1/2,g]$, then we have $y \in [-1/2,g)$ and thus 
\[
g^2 < 1+y \frac{q_{n-1}}{q_n} < 1+g
\]
by Lemma~\ref{q-n-2}.
If $\alpha \in (g,1]$, then we have $y \in (-g^2,1]$, with $y > 0$ if $a_n = 1$, thus 
\[
\frac{1}{2} < 1+y \frac{q_{n-1}}{q_n} < 3. \qedhere
\]
\end{proof}

\begin{prop} 
There exists an invariant probability measure~$\nu$ for $\overset{\circ}{T}_{\alpha}$ that is equivalent to the Lebesgue measure. 
\end{prop}

\begin{proof}
For any cylinder set $J$ of length $n$ such that $T_{\alpha}^{n}J = [\alpha-1, \alpha)$, the size of  $J$ is 
\[
\left|\frac{p_{n-1}(\alpha) T_{\alpha}^{n}(\alpha) + p_{n}(\alpha)}
     {q_{n-1}(\alpha) T_{\alpha}^{n}(\alpha) + q_{n}(\alpha)} 
-  \frac{p_{n-1}(\alpha-1) T_{\alpha}^{n}(\alpha-1) + p_{n}(\alpha-1)}
     {q_{n-1}(\alpha-1) T_{\alpha}^{n}(\alpha-1) + q_{n}(\alpha-1)} 
     \right|
\] 
From Lemma~\ref{q-n-2} and Proposition~\ref{bdd-dis},  this is $\sim q_{n}^{2}$ since 
the condition on~$J$ implies that $|a_{n}| \ge 2$.   
Then there exists a constant $C_1 > 1$ such that 
for any measurable set $A \subset [\alpha-1, \alpha)$ 
\[
C_1^{-1} \lambda(A) < 
\lambda\big(\overset{\circ}{T}_{\alpha}{}^{-m}(A)\big) <
C_1 \lambda (A) . 
\]
By the Dunford--Miller theorem we have that
\[
\mu_0(A) = \lim_{M\to \infty} \frac{1}{M} \sum_{m=1}^{M} \lambda\big(\overset{\circ}{T}_{\alpha}{}^{-m}(A)\big)
\]
exists for any measurable subset~$A$. 
It follows from the above estimate that 
\[
C_1^{-1} \lambda(A) \le \mu_{0}(A) \le C_1 \lambda(A),
\]
hence $\mu_0$ is a finite measure which is equivalent to 
Lebesgue measure.
\end{proof}

\begin{prop}  \label{ergodicity}
The map $\overset{\circ}{T}_{\alpha}$ is ergodic w.r.t.\ the Lebesgue measure.
\end{prop}

\begin{proof}
Suppose that $A$ is an invariant set of 
$\overset{\circ}{T}_{\alpha}$ with $\lambda(A) > 0$.  
For any $\varepsilon > 0$, there exists 
a full-cylinder set $J$ of length $n$ such that 
\[
\frac{\lambda (A\cap J)}{\lambda(J)} > 1 -\varepsilon.
\]
Then we see that there exists a constant $C_{2}>0$ such that  
\[
\overset{\circ}{T}_{\alpha}{\!\!\!}^{n}(A\cap J) > 1 - C_{2}\varepsilon
 \qquad  
\mbox{and} \qquad A \supset  \overset{\circ}{T}_{\alpha}{\!\!\!}^{n}(A\cap J) .  
\]
This shows $\lambda (A) = 1$. 
\end{proof}

We can now prove the ergodicity of $T_\alpha$. 

\begin{proof}[Proof of Theorem~\ref{t:ergodic}]
We refer to \cite{Schw} for determining the absolutely  
continuous invariant measure for $T_{\alpha}$ from 
that of $\overset{\circ}{T}_{\alpha}$ and the fact that 
the ergodicity of $\overset{\circ}{T}_{\alpha}$ implies that of 
$T_{\alpha}$. Indeed we put 
\begin{equation}  \label{mu-0}
\mu_{0}(A) = \sum_{n=0}^{\infty} 
\nu(T_{\alpha}^{-n} A \cap B_{n})
\end{equation}
which is an invariant measere for $T_{\alpha}$. Then 
the porperty 
\[
\sum_{n=1}^{\infty} \lambda(B_{n}) < \infty , 
\]
see \eqref{non-full-size}, ensures the finiteness
of the absolutely continuous invariant measure.  
Hence we have the invariant probability measure $\mu$ 
by normalization of $\mu_{0}$. Since $\mu$ is equivalent 
to $\nu$,  it is equivalent to the Lebesgue measure ~$\lambda$. 
Thus from Proposition~\ref{ergodicity}, it is easy to see that $T_{\alpha}$ is ergodic w.r.t.~$\mu$.   
\end{proof}

\begin{cor}
The map $T_{\alpha}$ is exact w.r.t.~$\mu$, i.e.\ the $\sigma$-algebra $\cap_{n=0}^{\infty} T_{\alpha}^{-n} \mathfrak{B}$ consists of sets of $\mu$-measures $0$ and~$1$.   
\end{cor}

\begin{proof}
For any interval $I \subset [\alpha - 1, \alpha)$, we
have 
\[
\lim_{n \to \infty} T_{\alpha}^{n}(I) 
= [\alpha - 1, \alpha).
\]
Indeed, from the proof of Proposition~\ref{p:full}, we can 
choose an innner point $x$ of $I$ so that 
$\langle a_{1}(x), a_{2}(x), \ldots , a_{n}(x) \rangle$ 
is a full cylinder.  This shows the assertion of this 
corollary; see~\cite{Ro}. 
\end{proof}

\begin{rem}
It is possible to show that $T_{\alpha}$ is weak Bernoulli following the idea of the 
proof by R.~Bowen \cite{Bow}, and the proof is similar to the case of other $\alpha$-continued fraction maps; see~\cite{N-N}. 
\end{rem}

\section{Planar natural extension}
We consider the planar natural extension map
\[
\mathcal{T}_\alpha:\, (x, y) \mapsto \bigg(\frac{1}{x} - \bigg\lfloor \frac{1}{x} + 1-\alpha\bigg\rfloor, \frac{1}{y+\big\lfloor \frac{1}{x} + 1-\alpha\big\rfloor}\bigg),
\]
with $\mathcal{T}_\alpha(0,y) = (0,0)$, and the natural extension domain
\[
\Omega_\alpha = \bigcup_{n\ge0} \overline{\mathcal{T}_\alpha^n\big([\alpha{-}1, \alpha) \times \{0\}\big)}.
\]
It is well known that $\Omega_1 = [0,1]^2$. 
It is easy to see that $\left(\Omega_{\alpha}, \mathcal T_{\alpha}, \frac{dx\,dy}{(1 + xy)^{2}}\right)$ is a natural extension of~$T_{\alpha}$ if $\Omega_{\alpha}$ has positive (two-dimensional) Lebesgue measure; see \cite[Theorem~1]{K-S-S}.
The invariance of the measure $\hat{\mu}$ given by $\mathrm{d} \hat{\mu} = \frac{\mathrm{d}x\,\mathrm{d}y}{(1 + xy)^{2}}$ 
is proved in the same way as those in \cite{Na-1, T-I}.  

The shape of $\Omega_\alpha$ was determined by Tanaka and Ito~\cite{T-I} for 
$\alpha \in [1/2, g]$.
In particular, we have 
\begin{equation} \label{e:Omegag}
\Omega_g = [-g^2,g^2] \times [1-\sqrt{2},\tfrac{1}{\sqrt{2}}-1] \cup [-g^2,g] \times [\tfrac{1}{\sqrt{2}}-1,2-\sqrt{2}];
\end{equation}
see Figure~\ref{f:natext}.  
The main purpose of this section is to prove that $\Omega_\alpha$ has positive measure for $\alpha > g$.  
To this end, we show that $\Omega_\alpha$ is contained in a certain polygon~$X_\alpha$, and then we relate $\Omega_\alpha$ to~$\Omega_g$.

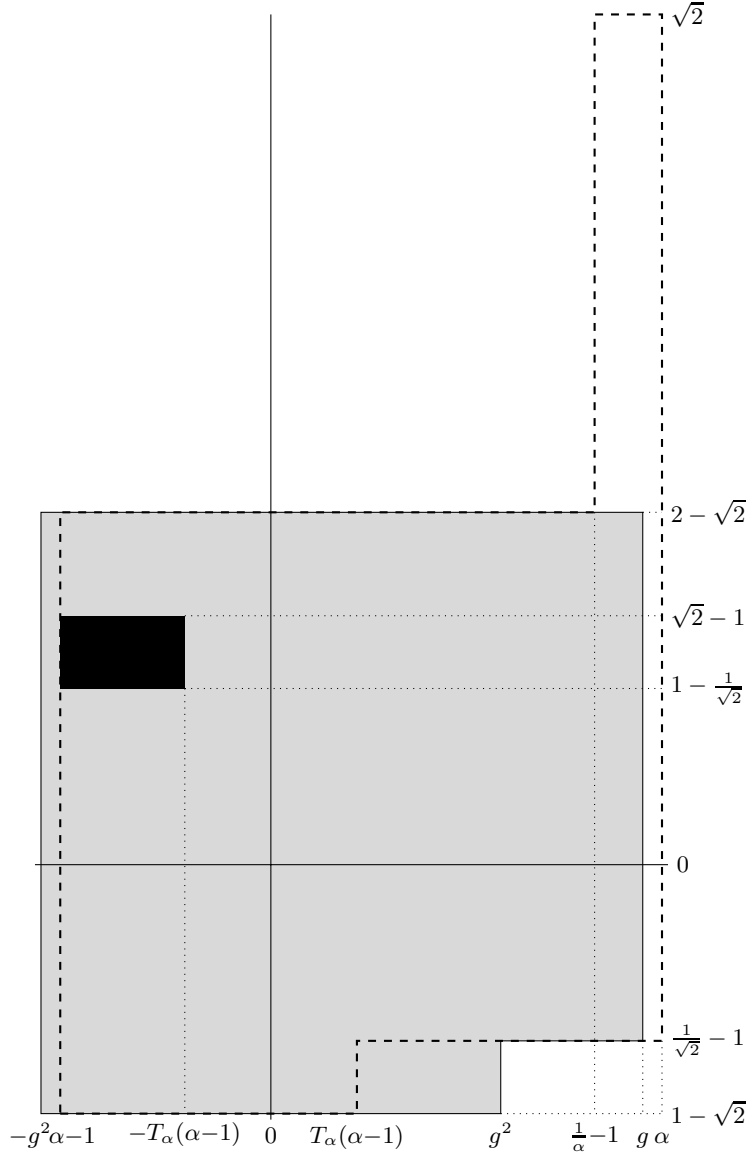
\begin{figure}[ht]
\begin{tikzpicture}[scale=8]
\small
\filldraw[fill=black!15](-.382,.-.414)--(.382,-.414)--(.382,-.293)--(.618,-.293)--(.618,.586)--(-.382,.586)--cycle;
\draw[thick,dashed](-.35,-.414)--(.143,-.414)--(.143,-.293)--(.65,-.293)--(.65,1.414)--(.538,1.414)--(.538,.586)--(-.35,.586)--cycle;
\fill[black](-.35,.293)--(-.143,.293)--(-.143,.414)--(-.35,.414);
\draw(0,-.424)--(0,1.414) (-.392,0)--(.66,0);
\draw[dotted](.618,-.293)--(.618,-.414) (.382,-.414)--(.65,-.414) (.618,.586)--(.65,.586) (-.143,.293)--(-.143,-.414) (.65,-.293)--(.65,-.414) (-.143,.293)--(.65,.293) (-.143,.414)--(.65,.414) (.538,.586)--(.538,-.414);
\node[below] at (0,-.414){$\vphantom{g^2}0$};
\node[below] at (-.4,-.414) {${-}g^2$};
\node[below] at (.382,-.414) {$g^2$};
\node[below] at (.618,-.414) {$\vphantom{g^2}g$};
\node[below] at (-.33,-.414) {$\vphantom{g^2}\alpha{-}1$};
\node[below] at (-.143,-.414) {${-}T_\alpha(\alpha{-}1)$};
\node[below] at (.143,-.414) {$\vphantom{g^2}T_\alpha(\alpha{-}1)$};
\node[below] at (.538,-.414) {$\frac{1}{\alpha}{-}1$};
\node[below] at (.65,-.414) {$\vphantom{g^2}\alpha$};
\node[right] at (.65,.586){$2-\sqrt{2}$};
\node[right] at (.66,0){$0$};
\node[right] at (.65,-.293){$\frac{1}{\sqrt{2}}-1$};
\node[right] at (.65,-.414){$1-\sqrt{2}$};
\node[right] at (.65,1.414){$\sqrt{2}$};
\node[right] at (.65,.293){$1-\frac{1}{\sqrt{2}}$};
\node[right] at (.65,.414){$\sqrt{2}-1$};
\end{tikzpicture}
\caption{The natural extension domain $\Omega_g$ is in grey; for $\alpha=13/20$, $\Omega_\alpha$ is contained in the dashed polygon $X_\alpha$ and contains the black rectangle.} \label{f:natext}
\end{figure}

\begin{lem} \label{l:Xalpha}
Let $\alpha \in (g,1)$ and $d = -a_1(\alpha-1)$.
We have $\Omega_\alpha \subset X_\alpha$ with
\[
X_\alpha = \big[\alpha-1,T_\alpha(\alpha-1)\big] \times \big[\tfrac{1}{2-\sqrt{2}-d},\tfrac{1}{1-\sqrt{2}-d}\big] \cup \big[\alpha-1,\alpha\big] \times \big[\tfrac{1}{1-\sqrt{2}-d},2-\sqrt{2}\big] \cup \big[\tfrac{1}{\alpha}-1, \alpha\big] \times \big[2-\sqrt{2},\sqrt{2}\big].
\]
\end{lem}

\begin{proof}
We see that $\mathcal{T}_\alpha(X_\alpha) \subset X_\alpha$ by determining the images of rectangles
\begin{align*}
\mathcal{T}_\alpha\Big(\big[\alpha-1,\tfrac{1}{\alpha-d-1}\big] \times \big[1-\sqrt{2},2-\sqrt{2}\big]\Big) & = \big[\alpha-1,T_\alpha(\alpha-1)\big] \times \big[\tfrac{1}{2-\sqrt{2}-d},\tfrac{1}{1-\sqrt{2}-d}\big], \\
\mathcal{T}_\alpha\Big(\big(\tfrac{1}{\alpha-d-1}, 0\big) \times \big[1-\sqrt{2},2-\sqrt{2}\big]\Big) & = \big[\alpha-1,\alpha\big) \times \big[\tfrac{1}{1-\sqrt{2}-d},0\big), \\
\mathcal{T}_\alpha\Big(\big(0, \tfrac{1}{\alpha+2}\big] \times \big[1-\sqrt{2}, \sqrt{2}\big]\Big) & = \big[\alpha-1,\alpha\big) \times \big(0, \tfrac{1}{4-\sqrt{2}}\big], \\
\mathcal{T}_\alpha\Big(\big(\tfrac{1}{\alpha+2},\tfrac{1}{\alpha+1}\big] \times \big[\tfrac{1}{\sqrt{2}}-1,\sqrt{2}\big]\Big) & = \big[\alpha-1,\alpha\big) \times \big[1-\tfrac{1}{\sqrt{2}}, 2-\sqrt{2}\big], \\
\mathcal{T}_\alpha\Big(\big(\tfrac{1}{\alpha+1}, \alpha\big] \times \big[\tfrac{1}{\sqrt{2}}-1,\sqrt{2}\big]\Big) & = \big[\tfrac{1}{\alpha}-1,\alpha\big) \times \big[\sqrt{2}-1, \sqrt{2}\big],
\end{align*}
and by using that $\frac{1}{2-\sqrt{2}-d} = \frac{-1}{1+\sqrt{2}} = 1-\sqrt{2}$ if $d=3$, $\frac{1}{2-\sqrt{2}-d} \ge \frac{-1}{2+\sqrt{2}} = \frac{1}{\sqrt{2}}-1$ if $d \ge 4$, $T_\alpha(\alpha-1) = \frac{1}{\alpha-1} + 3 < \frac{1}{\alpha+2}$ if $d=3$, and $\frac{1}{4-\sqrt{2}} < \sqrt{2}-1$. 
This implies that $\Omega_\alpha \subset X_\alpha$.
\end{proof}

We establish a relation between $\alpha$-expansions for different $\alpha$; see also ~\cite{C-L-S}.

\begin{lem} \label{l:alphabeta}
Let $g \le \alpha \le \beta \le 1$, $x \in [\alpha-1,\alpha)$, $z \in [\beta-1,\beta)$. 
\begin{enumerate}
\item \label{i:alphabeta1}
If $x=z$ or $(x+1)(1-z) = 1$ or $(1-x)(z+1) = 1$, then $T_\beta(z) - T_\alpha(x) \in \{0,1\}$.
\item \label{i:alphabeta2}
If $x+z = 0$ or $(x+1)(z+1) = 1$, then $T_\alpha(x) + T_\beta(z) \in \{0,1\}$.
\item \label{i:alphabeta4}
If $z-x =1$, then $(x+1) (T_\beta(z)+1) = 1$.
\item \label{i:alphabeta3}
If $x+z = 1$, then 
\[
\begin{cases}  
(T_\alpha(x)+1)(1-z) = 1 & \text{if } x > \frac{1}{\alpha+1}, \\  
(1-x)(T_\beta(z)+1) = 1 & \text{if } z > \frac{1}{\beta+1}, \\
\big(T_\alpha(x)+1\big) \big(T_\beta(z)+1\big) = 1 & \text{otherwise}.
\end{cases}
\] 
\end{enumerate}
\end{lem}

\begin{proof}
In case (\ref{i:alphabeta1}), we have $\frac{1}{x} - \frac{1}{z} \in \{-1,0,1\}$ or $x=z=0$, thus $T_\beta(z)-T_\alpha(x) \in \mathbb{Z}$.
We clearly have $T_\beta(z)-T_\alpha(x) \in (\beta-\alpha-1, \beta-\alpha+1) \subset (-1,2-g)$, thus $T_\beta(z) - T_\alpha(x) \in \{0,1\}$. 

In case~(\ref{i:alphabeta2}), $\frac{1}{x} + \frac{1}{z} \in \{-1,0\}$ or $x=z=0$ gives that $T_\alpha(x) + T_\beta(z) \in \mathbb{Z} \cap [\alpha+\beta-2, \alpha+\beta) = \{0,1\}$.

In case (\ref{i:alphabeta4}), we have $z = x+1 \ge \alpha \ge g$, thus $T_\beta(z) = \frac{1}{z}-1$ and $(x+1) (T_\beta(z)+1) = 1$. 

Finally, in case (\ref{i:alphabeta3}), if $x > \frac{1}{\alpha+1}$, then $T_\alpha(x) = \frac{1}{x}-1$ and $(T_\alpha(x)+1)(1-z) = 1$. 
Similarly, $z > \frac{1}{\beta+1}$ implies that $(1-x)(T_\beta(z)+1) = 1$.
If $x \le \frac{1}{\alpha+1}$ and $z \le \frac{1}{\beta+1}$, then $x = 1-z \ge \frac{\beta}{\beta+1}\ge \frac{1}{g+2} \ge \frac{1}{\alpha+2}$ and $z = 1-x \ge \frac{\alpha}{\alpha+1} \ge \frac{1}{g+2} \ge \frac{1}{\beta+2}$.
We cannot have $x = \frac{1}{\alpha+2}$ because this would imply that $\alpha = g = \beta = z$, contradicting that $z < \beta$. 
Similarly, we cannot have $z = \frac{1}{\beta+2}$.
From $x \in (\frac{1}{\alpha+2}, \frac{1}{\alpha+1}]$ and $z \in (\frac{1}{\beta+2}, \frac{1}{\beta+1}]$, we infer that $(T_\alpha(x)+1) (T_\beta(z)+1) = (\frac{1}{x} - 1) (\frac{1}{z} - 1) = 1$. 
\end{proof}

\begin{lem} \label{l:alphabeta2}
Let $g \le \alpha < \beta \le 1$, $x \in [\alpha-1,\alpha)$, $z \in [\beta-1,\beta)$, with $z-x \in \{0,1\}$ or $x+z \in \{0,1\}$. 
Let $n \ge 1$ be such that $T_\alpha^{n-1}(x) < \frac{\beta}{\beta+1}$.
Then there is some $k \ge 1$ such that $T_\beta^k(z) - T_\alpha^n(x) \in \{0,1\}$ or  $T_\alpha^n(z) + T_\beta^k(x) \in \{0,1\}$.
\end{lem}

\begin{proof}
Denote $x_j = T_\alpha^j(x)$ and $z_j = T_\beta^j(z)$.
By Lemma~\ref{l:alphabeta} and since $x_{n-1} < \frac{\beta}{\beta+1} < \frac{1}{\alpha+1}$, we have $z_k - x_n \in \{0,1\}$ or $x_n+z_k \in \{0,1\}$ or $(x_n+1)(z_k+1) = 1$ for some $k \ge 1$. 
If $(x_n+1)(z_k+1) = 1$, then $z_{k-1}-x_n=1$ (and $k \ge 2$) because $1-x_{n-1} = z_{k-1} \le \frac{1}{\beta+1}$ would contradict $x_{n-1} < \frac{\beta}{\beta+1}$.
\end{proof}

Define
\[
S(x,y) = \{(x,y), (-x,-y), (x+1,\tfrac{y}{1-y}), (1-x, \tfrac{-y}{y+1})\}.
\]

\begin{lem} \label{l:alphabeta3}
Let $g \le \alpha < \beta \le 1$, $(x,y) \in \Omega_\alpha$, $(\tilde{x},\tilde{y}) \in S(x,y)$, $(x_n,y_n) = \mathcal{T}_\alpha^n(x,y)$ for some $n \ge 1$.
If $\tilde{x} \in [\beta-1,\beta)$ and $y_n < 1-\frac{1}{\sqrt{2}}$, then there is some $k \ge 1$ such that $\mathcal{T}_\beta^k(\tilde{x},\tilde{y}) \in S(x_n,y_n)$. 
\end{lem}

\begin{proof}
Since $\Omega_\alpha \subset X_\alpha$ by Lemma~\ref{l:Xalpha}, $y_n < 1-\frac{1}{\sqrt{2}}$ implies that $a_n(x) \ge 3$ or $a_n(x) < 0$, i.e., $T_\alpha^{n-1}(x) \le \frac{1}{\alpha+2} < \frac{\beta}{\beta+1}$. 
Therefore, by Lemma~\ref{l:alphabeta2}, we have some $k \ge 1$ such that $T_\beta^k(\tilde{x}) - T_\alpha^n(x) \in \{0,1\}$ or  $T_\alpha^n(\tilde{z}) + T_\beta^k(x) \in \{0,1\}$.
Considering the associated linear fractional transformations, we obtain that $\mathcal{T}_\beta^k(\tilde{x},\tilde{y}) \in S \mathcal{T}_\alpha^n(x,y)$.
\end{proof}

\begin{lem} \label{l:alphabeta4}
Let $g \le \alpha < \beta \le 1$, $(x,y) \in \Omega_\alpha$ with $y < 1-\frac{1}{\sqrt{2}}$.
Then we have $S(x,y) \cap \Omega_\beta \ne \emptyset$.
\end{lem}

\begin{proof}
Assume first that $(x,y) = \mathcal{T}_\alpha^n(z,0)$ for some $n \ge 0$, $z \in [\alpha-1,\alpha)$, and choose $\tilde{z} \in [\beta-1,\beta)$ such that $(\tilde{z},0) \in S(z,0)$.
Since $y < 1-\frac{1}{\sqrt{2}}$, Lemma~\ref{l:alphabeta3} gives some $k \ge 0$ such that $\mathcal{T}_\alpha^k(\tilde{z},0) \in S(x,y)$, thus $S(x,y) \cap \Omega_\beta \ne \emptyset$. 
As each $(x,y) \in \Omega_\alpha$ is the limit of points $\mathcal{T}_\alpha^n(z,0)$, this proves the lemma.
\end{proof}

From Lemma~\ref{l:alphabeta4} with $\alpha=g$, we can easily conclude that $\Omega_\beta$ has positive Lebesgue measure, and the following lemma provides rectangles in the natural extension domain.

\begin{lem} \label{l:rectangles}
Let $\alpha \in (g,1)$, $d = -a_1(\alpha-1)$, $b = \lfloor T_\alpha(\alpha-1) + \alpha\rfloor$.
We have $Y_\alpha \subset \Omega_\alpha$, with 
\[
Y_\alpha = \big[\alpha-1,b-T_\alpha(\alpha-1)\big] \times \big[\tfrac{1}{d+\sqrt{2}-1-b},\tfrac{1}{d+\sqrt{2}-2-b}\big] \cup \big[\alpha-1,\alpha\big] \times \big(\tfrac{1}{d+\sqrt{2}-2-b},\sqrt{2}-1\big].
\]
\end{lem}

\begin{proof}
Let $(x,y) \in \Omega_g \setminus \Omega_\alpha$ with $y<0$. 
Then Lemma~\ref{l:alphabeta4} gives that $(-x,-y) \in \Omega_\alpha$ or $(x+1,\frac{y}{1-y}) \in \Omega_\alpha$ or $(1-x,\frac{-y}{y+1}) \in \Omega_\alpha$.
We have thus $(-x,-y) \in \Omega_\alpha$ when $|x|<1-\alpha$,  $(1-x,\frac{-y}{y+1}) \in \Omega_\alpha$ when $x>1-\alpha$.
If $x\le\alpha-1$ and $y<\frac{1}{2-\sqrt{2}-d}$, then we also have $(-x,-y) \in \Omega_\alpha$ because $x+1 \ge g$ and $\frac{y}{1-y} < \frac{1}{1-\sqrt{2}-d}$ imply that $(x+1,\frac{y}{1-y}) \notin \Omega_\alpha$ by Lemma~\ref{l:Xalpha}.

From Lemma~\ref{l:Xalpha} and equation~\eqref{e:Omegag}, we get that 
\[
\big([{-}g^2,g] \times \big[1-\sqrt{2},\tfrac{1}{2-\sqrt{2}-d}\big) \cup (T_\alpha(\alpha-1),g] \times \big[\tfrac{1}{2-\sqrt{2}-d},\tfrac{1}{1-\sqrt{2}-d}\big)\big) \setminus (g^2,g] \times \big[1-\sqrt{2},\tfrac{1}{\sqrt{2}}-1\big) \subset \Omega_g \setminus \Omega_\alpha.
\]
Considering points $(x,y)$ with $x<1-\alpha$ in this union of rectangles, we obtain that
\[
\big(\alpha-1,g^2\big] \times \big(\tfrac{1}{d+\sqrt{2}-2},\sqrt{2}-1\big] \cup \big(\alpha-1,\max\{\alpha-1,-T_\alpha(\alpha-1)\}\big) \times \big(\tfrac{1}{d+\sqrt{2}-1},\tfrac{1}{d+\sqrt{2}-2}\big] \subset \Omega_\alpha.
\]
If $d \ge 4$, then points $(x,y)$ with $x>1-\alpha$ and $y \ge \frac{1}{\sqrt{2}}-1$ provide that
\[
\big[g^2,\alpha\big) \times \big(\tfrac{1}{d+\sqrt{2}-3},\sqrt{2}-1\big] \cup \big[g^2,\min\{\alpha,1-T_\alpha(\alpha-1)\}\big) \times \big(\tfrac{1}{d+\sqrt{2}-2},\tfrac{1}{d+\sqrt{2}-3}\big] \subset \Omega_\alpha.
\]
By distinguishing the cases $T_\alpha(\alpha-1) < 1-\alpha$, i.e., $b=0$, and $T_\alpha(\alpha-1) \ge 1-\alpha$, i.e., $b=1$, we get that $Y_\alpha \subset \Omega_\alpha$. 
(Note that $\Omega_\alpha$ is a closed set.) 
\end{proof}

Since for $\alpha \in (g,1)$ we have $d \ge 3$, with $b=0$ if $d=3$, Lemma~\ref{l:rectangles} shows in particular that
\begin{equation} \label{e:rectangle}
\big[\alpha-1, \min\{\alpha, \tfrac{1}{1-\alpha}-3\}\big] \times \big[1-\tfrac{1}{\sqrt{2}}, \sqrt{2}-1\big] \subset \Omega_\alpha
\end{equation}
(with $\frac{1}{1-\alpha}-3>\alpha-1$).
Theorem~\ref{t:natext} is a direct consequence of this inclusion.

\section{Entropy}

From \eqref{e:rectangle}, we obtain the following proposition. 

\begin{prop} \label{p:bound}
There exists a positive constant $C_3$ such that 
\[
C_3^{-1} \lambda(A) < \mu_\alpha(A) < C_3 \lambda(A)
\]
for any measurable set $A \subset [\alpha-1, \alpha)$. 
\end{prop}

\begin{proof}  
By Proposition~\ref{p:full}, we have a full cylinder $\langle a_{1}(x), \ldots , a_{n}(x)\rangle \subset \big[\alpha-1, \min\{1-\alpha, \tfrac{1}{1-\alpha}-3\}\big]$. 
Then there exists a real number~$y_{0}$ and a positive number~$\eta$ such that 
\[
\mathcal{T}_{\alpha}^{n}\big(\langle a_{1}(x),\ldots,a_{n}(x)\rangle \times \big[1-\tfrac{1}{\sqrt{2}}, \sqrt{2}-1\big] \big)= [\alpha -1, \alpha) 
\times [y_{0}, y_{0} + \eta]. 
\]
This shows that there is a positive constant $C_3'$ such that $\xi(x) > C_3'$, where 
\[
\xi(x) = \frac{1}{\hat{\mu}(\Omega_{\alpha})} \int_{y :\, (x, y) \in \Omega_{\alpha}} 
\frac{1}{(1 + xy)^{2}}\, \mathrm{d}y
\]
is the density of $\mu_\alpha$.
On the other hand, since $\Omega_{\alpha} \subset [\alpha-1, \alpha] \times [1- \sqrt{2}, \sqrt{2}]$, we can find $C_3''$ such that $\xi(x) < C_3''$. 
Altogether, we have the assertion of this proposition. 
\end{proof}

Let $h(T_{\alpha})$ denote the entropy of~$T_{\alpha}$ with respect to the invariant measure~$\mu_\alpha$. 
The following shows that Rokhlin's formula holds, as mentioned at the end of~\S3. 

\begin{prop} \label{p:rokhlin}
For any $0 < \alpha \le 1$, we have 
\[
h(T_{\alpha}) = - \int_{[\alpha -1, \alpha)} \log x^{2}\, \mathrm{d}\mu_\alpha(x)  
\]
and 
\[
h(T_{\alpha}) = 2 \lim_{n \to \infty} \frac{1}{n} \log |q_{n}(x)| \quad \mbox{for a.e.}\ x \in [\alpha-1, \alpha).
\]
\end{prop}

\begin{proof}
Choose a generic point $x \in [\alpha-1, \alpha)$ so that 
\begin{itemize}
\item
there exists subsequence of natural numbers $(n_{k})_{k\ge 1}$ such that $\langle a_{1}(x), \ldots , a_{n_{k}}(x)\rangle$ is a full cylinder for any $k \ge 1$,
\item
$- \lim_{n \to \infty} \frac{1}{n}\log \mu_\alpha (\langle a_{1}(x), \ldots , a_{n}(x)\rangle) = h(T_{\alpha})$, \item
$-\lim_{N \to \infty} \sum_{n=0}^{N} \log (T_{\alpha}x)^{2} = -\int_{[\alpha -1, \alpha)} \log x^{2}\, \mathrm{d}\mu_\alpha$.  
\end{itemize}
For each $n_{k}$, we see that
\begin{equation} \label{lambda} 
\lambda (\langle a_{1}(x), \ldots , a_{n_{k}}(x)\rangle)
= 
\left| \frac{p_{n-1} \cdot \alpha + p_{n}}
       {q_{n-1} \cdot \alpha + q_{n}} 
       - 
 \frac{p_{n-1} \cdot (\alpha-1) + p_{n}}
       {q_{n-1} \cdot (\alpha-1) + q_{n}}       
\right| .      
\end{equation}
From Proposition~\ref{p:bound}, we have 
\[
\lim_{k \to \infty} \frac{1}{n_{k}} 
\log \mu_\alpha(\langle a_{1}(x), \ldots , a_{n_{k}}(x)\rangle)
= 
\lim_{k \to \infty} \frac{1}{n_{k}} 
\log \lambda(\langle a_{1}(x), \ldots , a_{n_{k}}(x)\rangle).
\]
Then by the mean-value theorem and \eqref{lambda} there 
exists $y_{0} \in [\alpha-1, \alpha)$ such that  
\[
h(T_{\alpha}) = -\lim_{k \to \infty} 
\frac{1}{n_{k}} \log \left| \psi_{a_{1}(x) \cdots a_{n_{k}}(x)}' (y_{0})\right| . 
\]
From Proposition \ref{bdd-dis}, we see 
\[
h(T_{\alpha}) = -\lim_{k \to \infty} 
\frac{1}{n_{k}} \log \left| \psi_{a_{1}(x) \cdots a_{n_{k}}(x)}' (y)\right|  
\]
for any $y \in [\alpha-1, \alpha)$.  So we can choose 
$y = T_{\alpha}(x)$.  
Then 
\[
\psi_{a_{1}(x) \cdots a_{n_{k}}(x)}' (y) 
= \frac{1}{\left(T_{\alpha}^{n_{k}}\right)'(x)}
\] 
holds.  
Consequently by the choice of $x$ and the chain rule we have the first assertion of this proposition. 
The second assertion also follows from Proposition~\ref{bdd-dis}. 
\end{proof}

Finally, we establish the monotonicity of the product $h(T_\alpha) \hat{\mu}(\Omega_\alpha)$.

\begin{proof}[Proof fo Theorem~\ref{t:hmu}]
For each $\alpha \in [1/2,g]$, we have $h(T_\alpha) = \frac{\pi^2}{6}$ and 
$\hat{\mu}(\Omega_\alpha) = -2 \log g$. 
Let now $g \le \alpha < \beta \le 1$, $d = -a_1(\alpha-1)$, $b = \lfloor T_\alpha(\alpha-1) + \alpha\rfloor$.
Set
\[
X_{\alpha,\beta} = \begin{cases}
\big(\max\{1{-}\beta,\tfrac{1}{\beta-1}+d+1\},\alpha\big) \times \big[\tfrac{1}{1-\sqrt{2}-d},\tfrac{1}{-\sqrt{2}-d}\big] \cap \Omega_\alpha & \mbox{if}\ T_\alpha(\alpha{-}1) = \alpha{-}1, \\ \big(\max\{\alpha{-}1,\tfrac{1}{\beta-1}+d\},T_\alpha(\alpha{-}1)\big] \times \big[\tfrac{1}{2-\sqrt{2}-d},\tfrac{1}{1-\sqrt{2}-d}\big] \cap \Omega_\alpha & \mbox{if}\ \alpha{-}1 {\,<\,} T_\alpha(\alpha{-}1) {\,\le\,} 1{-}\beta, \\ 
\big(\max\{1{-}\beta,\tfrac{1}{\beta-1}+d\},T_\alpha(\alpha{-}1)\big] \times \big[\tfrac{1}{2-\sqrt{2}-d},\tfrac{1}{1-\sqrt{2}-d}\big] \cap \Omega_\alpha & \mbox{if}\ T_\alpha(\alpha{-}1) > 1{-}\beta.
\end{cases}
\]
Note that $X_{\alpha,\beta} \subset X_\alpha \setminus X_\beta$, and we have $\hat{\mu}(X_{\alpha,\beta})>0$ because of \eqref{e:rectangle} together with $\mathcal{T}_\alpha(\Omega_\alpha) \subset \Omega_\alpha$,
\[
\mathcal{T}_\alpha([\alpha-1,x] \times [1-\sqrt{2},2-\sqrt{2}]) = [T_\alpha(x),T_\alpha(\alpha-1)] \times [\tfrac{1}{2-\sqrt{2}-d},\tfrac{1}{1-\sqrt{2}-d}] \ \mbox{for all}\ x \in (\alpha-1,\tfrac{1}{\alpha-d-1}],
\]
and, in case $T_\alpha(\alpha-1) = \alpha-1$,
\[
\mathcal{T}_\alpha((\alpha-1,x] \times [1-\sqrt{2},2-\sqrt{2}]) = [T_\alpha(x),T_\alpha(\alpha-1)) \times [\tfrac{1}{1-\sqrt{2}-d},\tfrac{1}{-\sqrt{2}-d}] \ \mbox{for all}\ x \in (\alpha-1,\tfrac{1}{\alpha-d-2}].
\]
Let 
\[
\varphi(x,y) = \begin{cases}(-x,-y) & \mbox{if}\ T_\alpha(\alpha-1) \in (1-\alpha,1-\beta], \\ (1-x,\frac{-y}{y+1}) & \mbox{otherwise.}\end{cases}
\]
Then we have $\hat{\mu}(\varphi(X_{\alpha,\beta})) = \hat{\mu}(X_{\alpha,\beta})$ and, by Lemma~\ref{l:rectangles}, $\varphi(X_{\alpha,\beta}) \subset \Omega_\beta$. 
Let $\widetilde{\mathcal{T}}_\alpha$ be the first return map of $\mathcal{T}_\alpha$ on~$X_{\alpha,\beta}$, and let $\widetilde{\mathcal{T}}_\beta$ be the first return map of $\mathcal{T}_\beta$ on~$\varphi(X_{\alpha,\beta})$.
For $(x,y) \in X_{\alpha,\beta}$, we have, by Lemma~\ref{l:alphabeta3}, $\mathcal{T}_\beta^k\varphi(x,y) \in S\widetilde{\mathcal{T}}_\alpha(x,y)$ for some $k \ge 1$, thus $\mathcal{T}_\beta^k\varphi(x,y) = \varphi\widetilde{\mathcal{T}}_\alpha(x,y)$, hence $\varphi\widetilde{\mathcal{T}}_\alpha(x,y) = \widetilde{\mathcal{T}}_\beta^m\varphi(x,y)$ for some $m \ge 1$.
This implies that $h(\widetilde{\mathcal{T}}_\beta) \le h(\widetilde{\mathcal{T}}_\alpha)$. 
Abramov's formula gives that
\[
h(\widetilde{\mathcal{T}}_\alpha) = \frac{\hat{\mu}
(\Omega_\alpha)}{\hat{\mu}(X_{\alpha,\beta})} h(\mathcal{T}_\alpha) \quad \mbox{and} \quad h(\widetilde{\mathcal{T}}_\beta) = \frac{\hat{\mu}(\Omega_\beta)}{\hat{\mu}(\varphi(X_{\alpha,\beta}))} h(\mathcal{T}_\beta),
\]
thus $\hat{\mu}(\Omega_\beta)\, h(\mathcal{T}_\beta) \le \hat{\mu}(\Omega_\alpha)\, h(\mathcal{T}_\alpha)$.
\end{proof}

\end{document}